\documentclass[a4paper,11pt,leqno]{amsart}
\usepackage{amsmath}
\usepackage{amsthm}
\usepackage{amsfonts}
\usepackage{amssymb}

\newtheorem{theorem}{Theorem}[section]{\bf}{\it}
\newtheorem{mainth}{Theorem}{\bf}{\it}
\newtheorem{lemma}[theorem]{Lemma}{\bf}{\it}
{\bf}{\it}
{\bf}{\it}
{\bf}{\it}

\theoremstyle{definition}

\theoremstyle{remark}

\numberwithin{equation}{section}

\newcommand{\cE}{\mathcal{E}}

\newcommand{\bR}{\mathbb R}
\newcommand{\bC}{\mathbb C}
\newcommand{\bN}{\mathbb N}
\newcommand{\bZ}{\mathbb{Z}}

\newcommand{\loc}{{\operatorname{loc}}}

\newcommand{\bS}{\mathbb{S}}

\newcommand{\bM}{\mathbb{M}}

\newcommand{\R}{\mathbb R}
\newcommand{\cF}{\mathcal F}

\begin{document}

\title[Accumulation of periodic points]{Accumulation of periodic points for local uniformly quasiregular mappings}

\date{\today}

\author{Y\^usuke Okuyama}
\address{Division of Mathematics, Kyoto Institute of Technology, Sakyo-ku, Kyoto 606-8585, Japan
}
\email{okuyama@kit.ac.jp}
\thanks{Y. O. is partially supported by JSPS Grant-in-Aid for Young Scientists (B), 24740087.}

\author{Pekka Pankka}
\address{University of Helsinki, Department of Mathematics and Statistics (P.O. Box 68), FI-00014 University of Helsinki, Finland}
\email{pekka.pankka@helsinki.fi}
\thanks{P. P. is partially supported by the Academy of Finland project \#256228.}



\begin{abstract}
We consider accumulation of periodic points in local uniformly quasiregular dynamics. Given a local uniformly quasiregular mapping $f$ with a countable and closed set of isolated essential singularities and their accumulation points on a closed Riemannian manifold, we show that points in the Julia set are accumulated by periodic points. If, in addition, the Fatou set is non-empty and connected, the accumulation is by periodic points in the Julia set itself. We also give sufficient conditions for the density of repelling periodic points.
\end{abstract}

\maketitle

\section{Introduction}
\label{sec:intro}

Let $M$ and $N$ be oriented Riemannian $n$-manifolds for $n\ge 2$. A continuous mapping $f\colon M \to N$ is called \emph{$K$-quasiregular}, $K\ge 1$, if $f$ belongs to the Sobolev space $W^{1,n}_\loc(M,N)$ and satisfies the distortion inequality
\[
\|df\|^n \le K\det df \quad \mathrm{a.e.\ on}\ M,
\]
where $\|df\|$ is the operator norm of the differential $df$ of $f$. 

A quasiregular self-map $f\colon M\to M$ is called \emph{uniformly $K$-quasiregular ($K$-UQR)} if all iterates $f^k$ for $k\ge 1$ are $K$-quasiregular. Similarly as quasiregular mappings have the r\^ole of holomorphic mappings in the $n$-dimensional Euclidean conformal geometry for $n\ge 3$, the dynamics of uniformly quasiregular mappings can be viewed as the counterpart of holomorphic dynamics in the $n$-dimensional conformal geometry. We refer to the seminal paper of Iwaniec and Martin \cite{Iwaniec-Martin} and Hinkkanen, Martin, Mayer \cite{HMM04} for the fundamentals in this theory.

In this article we consider dynamics of local UQR-mappings. Let $M$ be an oriented Riemannian $n$-manifold and $\Omega\subset M$ an open set. Following the terminology in \cite{HMM04},
we say a mapping $f\colon \Omega\to M$ 
is a {\itshape local uniformly $K$-quasiregular}, $K\ge 1$, 
if for every $k\in\bN$, $\bigcap_{j=0}^{k-1}f^{-j}(\Omega)\neq\emptyset$ and
$f^k\colon \bigcap_{j=0}^{k-1}f^{-j}(\Omega)\to M$ is $K$-quasiregular. 

With slight modifications, the standard terminology from dynamics is at our disposal also in this local setting. Let
\begin{gather*}
 D_f:=\text{the interior of }\bigcap_{k\ge 0}f^{-k}(\Omega)
=M\setminus\overline{\bigcup_{k\ge 0}f^{-k}(M\setminus \Omega)}.
\end{gather*}
As usual, the Fatou set $F(f)$ of $f$ is 
the maximal open subset in $D_f$ where
the family $\{f^k;k\in\bN\}$ is normal, the Julia set of $f$ is the set
\begin{gather*}
 J(f):=M\setminus F(f),
\end{gather*}
and the exceptional set of $f$ is
\begin{gather*}
 \cE(f):=\{x\in M;\#\bigcup_{k\ge 0}f^{-k}(x)<\infty\}.
\end{gather*}

A point $x\in M$ is a {\itshape periodic point of $f$ in $M$} if $x\in\bigcap_{j=0}^{p-1}f^{-j}(\Omega)$ and $f^p(x)=x$ for some $p\in\bN$. We call $p$ a {\itshape period of $x$ (under $f$)}. Note that periodic points always belong to the set $\overline{D_f}$. 

A periodic point $x\in M$ with period $p\in \bN$ is {\itshape (topologically) repelling} if $f\colon U\to f^p(U)$ is univalent and $U\Subset f^p(U)$
for some open neighborhood $U$ of $x$ in $\bigcap_{j=0}^{p-1}f^{-j}(\Omega)$. 
Note that, then $x\in J(f)$; see \cite[\S 4]{HMM04}.

In \cite{HMM04}, Hinkkanen, Martin and Mayer 
gave a classification of
cyclic Fatou components of $f$ (see Theorem \ref{th:classification}) 
as well as periodic points.
We study both $J(f)$ and $\cE(f)$ for
a non-constant local uniformly quasiregular mapping
\begin{gather*}
 f\colon \bM\setminus S_f\to\bM,
\end{gather*}
where $\bM$ is a closed, oriented, and connected Riemannian $n$-manifold, $n\ge 2$, and $S_f$ is a countable and closed subset in $\bM$ 
consisting of isolated essential singularities of $f$ and their accumulation points in $\bM$. In our first main theorem, we also consider a sub-class of non-elementary UQR-mappings. A non-constant local uniformly quasiregular mapping $f\colon \bM\setminus S_f\to\bM$ is {\itshape non-elementary} if it is non-injective and satisfies
\begin{gather*}
 J(f)\not\subset\cE(f). 
\end{gather*}
For comments on the non-injectivity and non-elementarity, see Section \ref{sec:discussion}.

Recall that a point $x$ in a topological space $X$
is {\itshape accumulated by} a subset $S$ in $X$
if for every neighborhood $N$ of $x$, $S\cap(N\setminus\{x\})\neq\emptyset$,
and that a subset $S$ in $X$ is {\itshape perfect}
if $S$ is non-empty, compact, and has no isolated points in $X$.

\begin{mainth}\label{th:periodic}
Let $\bM$ be a closed, oriented, and connected Riemannian $n$-manifold, $n\ge 2$, and $f \colon \bM\setminus S_f \to\bM$ a non-constant local uniformly 
$K$-quasiregular mapping, $K\ge 1$, where $S_f$ is a countable and closed subset in $\bM$ and consists of isolated essential singularities of $f$ and their accumulation points in $\bM$. 
Then $J(f)$ is nowhere dense in $\bM$ unless $J(f)=\bM$. Furthermore, the following hold:

\begin{itemize}
\item[(a)] If $f$ is non-injective, then $J(f)\neq\emptyset$ and $\#\cE(f)<\infty$. Moreover, for every $x\in\bM\setminus\cE(f)$, points in $J(f)$ are accumulated by $\bigcup_{k\ge 0}f^{-k}(x)$.

\item[(b)] If $f$ is non-injective and $S_f=\emptyset$, then $\cE(f)\subset F(f)$ and $f$ is non-elementary. 

\item[(c)] If $f$ is {\itshape a priori} non-elementary, then $J(f)$ is perfect and points in $J(f)$ are accumulated by periodic points of $f$. 
\end{itemize}
\end{mainth}
For non-constant and non-injective
uniformly quasiregular endomorphisms of the $n$-sphere $\bS^n$,
the accumulation of periodic points to $J(f)$ in Theorem \ref{th:periodic} is 
due to Siebert \cite[3.3.6 Theorem]{Siebert06}; 
note that by a theorem of Fletcher and Nicks \cite{FletcherA:Julsuq}, 
$J(f)$ is in fact uniformly perfect in this case.

The proof of the accumulation of periodic points to the Julia set
for non-elementary $f$ is based on two rescaling principles 
(see Section \ref{sec:background}).
It is a generalization of Schwick's argument \cite{SchwickW:Reppp}
(see also Bargmann \cite{BargmannD:Simpsf} and Berteloot--Duval \cite{BertelootF:Demddc}), 
which is a reminiscent to Julia's construction of
(expanding) homoclinic orbits for rational functions (\cite[\S 14]{Milnor3rd}).
Our argument simultaneously treats all the cases
$S_f=\emptyset$, 
$0<\#\bigcup_{k\ge 0}f^{-k}(S_f)<\infty$, and 
$\# \bigcup_{k\ge 0}f^{-k}(S_f)=\infty$, 
which are typically studied separately.

In the final assertion in Theorem \ref{th:periodic},
it would be natural and desirable 
to obtain the density of (repelling) periodic points in $J(f)$.

Our second main theorem gives sufficient conditions for those density results.
The topological dimension of a subset $E$ in $\bM$ is denoted by $\dim E$ and the branch set of $f$ by $B_f$; the {\itshape branch set} $B_f$ is the set of points at which $f$ is not a local homeomorphism.

\begin{mainth}\label{th:repdense}
Let $\bM$ be a closed, oriented, and connected Riemannian $n$-manifold, $n\ge 2$, and $f\colon \bM\setminus S_f \to\bM$ be a non-elementary local uniformly 
$K$-quasiregular mapping, $K\ge 1$, where $S_f$ is a countable and closed subset in $\bM$ and consists of isolated essential singularities of $f$ and their accumulation points in $\bM$. Then

\begin{itemize}
 \item[(a)]  If $F(f)$ is non-empty and connected, then
 points in $J(f)$ are accumulated by periodic points of $f$ contained in $J(f)$.

 \item[(b)] If one of the following four conditions
 \begin{enumerate}
 \item  $\#\bigcup_{k\ge 0}f^{-k}(S_f)<\infty$ and $\dim J(f)>n-2$,
        \label{item:connected}
  \item $f$ has a repelling periodic point
       in $D_f\setminus(\cE(f)\cup\bigcup_{k\in\bN}f^k(B_{f^k}))$, 
       \label{item:rep}
 \item $J(f)\not\subset\bigcap_{j\in\bN}\overline{\bigcup_{k\ge j}f^k(B_{f^k})}$, or\label{item:pcs}
 \item $n=2$ \label{item:surface}
 \end{enumerate}
 holds, then points in $J(f)$ are accumulated by repelling periodic points of $f$.
\end{itemize}
\end{mainth}

Theorem \ref{th:repdense} combines and extends previous results of Hinkkanen--Martin--Mayer (\cite{HMM04}) and Siebert (\cite{SiebertThesis}) for UQR-mappings and classical results of Fatou and Julia (\cite[\S 14]{Milnor3rd}), Baker \cite{Baker68}, Bhattacharyya \cite{Bhattacharyya69}, and Bolsch \cite{Bolsch96} and Herring \cite{HerringThesis} in the holomorphic case.

For non-constant and non-injective uniformly quasiregular endomorphisms of $\bS^n$, the repelling density in $J(f)$ is due to Hinkkanen, Martin and Mayer \cite{HMM04} when $F(f)$ is either empty or not connected. Under these conditions $S_f=\emptyset$ and $\dim J(f)>n-2$. Siebert \cite[4.3.6 Satz]{SiebertThesis} proved the repelling density under the assumption $J(f)\not\subset\overline{\bigcup_{k\in\bN} f^k(B_{f^k}})$. In this case $J(f)\not\subset\bigcap_{j\in\bN}\overline{\bigcup_{k\ge j}f^k(B_{f^k})}$.

In the holomorphic dynamics, i.e.\;for $\bM=\bS^2$ (so $n=2$) and $K=1$,
every non-constant and non-injective holomorphic mapping 
$f\colon \bS^2\setminus S_f\to\bS^2$
is non-elementary (see Section \ref{sec:discussion}).
For $S_f=\emptyset$, the repelling density in $J(f)$ is a classical result of   
Fatou and Julia (cf.\ \cite[\S 14]{Milnor3rd}). For $\#\bigcup_{k\ge 0} f^{-k}(S_f)=1,2$ and $\#S_f =\infty$, it is due to Baker \cite{Baker68}, Bhattacharyya \cite{Bhattacharyya69}, Bolsch \cite{Bolsch96} and Herring \cite{HerringThesis}.
Note that our proof covers also the case $\#\bigcup_{k\ge 0} f^{-k}(S_f)>2$.

This paper is organized as follows. In Section \ref{sec:background}, 
we give a unified treatment for normal families and isolated essential singularities of quasiregular mappings. We also recall the invariance of
the dynamical sets $D_f,F(f),J(f)$, and $\cE(f)$ under $f$
and the Hinkkanen--Martin--Mayer classification for cyclic Fatou components
of non-elementary local uniformly quasiregular mappings. 
In Sections \ref{sec:theorem1} and \ref{sec:final},
we prove Theorems \ref{th:periodic} and \ref{th:repdense}. 
We finish, in Section \ref{sec:discussion}, with comments on the non-injectivity
and non-elementarity of non-constant local uniformly quasiregular dynamics.

\section{Preliminaries}
\label{sec:background}

We begin with notations and fundamental facts from the local degree theory.
For each oriented $n$-manifold $X$, we fix a generator $\omega_X$ of $H^n_c(X;\bZ)$
representing the orientation of $X$, and for each subdomain $D\subset X$, a generator $\omega_D$ of $H^n_c(D;\bZ)$ satisfying $\omega_X=\iota_{D,X}(\omega_D)$, 
where $\iota_{D,X}\colon H^n_c(D;\bZ) \to H^n_c(X;\bZ)$ is the canonical isomorphism.

Let $f \colon M \to N$ be a continuous mapping between oriented
$n$-manifolds $M$ and $N$.
For each domain $D\subset M$ and each $y\in N\setminus f(\partial D)$, 
the \emph{local degree of $f$ at $y\in N$ with respect to $D$} is the 
non-negative integer $\mu(y,f,D)$ satisfying
\begin{gather}
 \mu(y,f,D)\omega_D = \iota_{V, D}((f|V)^*\omega_{\Omega}),\label{eq:cohomology}
\end{gather}
where $\Omega$ is the component of $N\setminus f(\partial D)$ containing $y$ 
and $V=f^{-1}(\Omega)\cap D$. Indeed, we can take any open 
and connected neighborhood
of $y$ in $N\setminus f(\partial D)$ as $\Omega$. 
If $\mu(y,f,D)>0$, then $f^{-1}(y)\cap D\neq\emptyset$.
For more details, see e.g., \cite[Section I.2]{HeinonenJ:Geobcg}.

From now on, let $n\ge 2$ and $K\ge 1$. 
Let $M$ and $N$ be connected and oriented Riemannian $n$-manifolds,
and $f\colon M\to N$ a non-constant quasiregular mapping.
By Reshetnyak's theorem (see e.g.\;\cite[I.4.1]{Rickman93}), 
$f$ is a \emph{branched cover}, that is, an open and 
discrete mapping. 
Every $x\in M$ has a {\itshape normal neighborhood} with respect to $f$, that is, an open neighborhood $U$ of $x$ satisfying
$f(\partial U)= \partial(f(U))$ and $f^{-1}(f(x))\cap U = \{x\}$. 
We denote by $i(x,f)$ the \emph{topological index of $f$ at $x$}, that is, 
$i(x,f) = \mu(f(x),f,U)$.
The branch set $B_f$ of $f$ is the set of all $x\in M$
satisfying $i(x,f)\ge 2$, and is closed in $M$.
By the Chernavskii-V\"ais\"al\"a theorem \cite{VaisalaJ:Disomm}, 
the topological dimensions $\dim B_f$ and $\dim f(B_f)$ are at most $n-2$.

The local degree theory readily yields the following manifold version of
the Miniowitz--Rickman argument principle or the Hurwitz-type theorem; 
see \cite[Lemma 2]{Miniowitz82}; note that we do not assume that mappings $f_j$ to be quasiregular.

\begin{lemma}\label{lemma:Hurwitz}
Let $M$ and $N$ be oriented Riemannian $n$-manifolds, $n\ge 2$.  
Suppose a sequence $(f_j)$ of continuous mappings  
from $M$ to $N$ tends to a quasiregular mapping $f\colon M\to N$ locally uniformly 
on $M$ as $j\to\infty$.
Then for every domain $D\Subset M$ with $f(\partial D)=\partial(f(D))$  
and every compact subset $E\subset N\setminus f(\partial D)$, there  
exists $j_0\in\bN$ such that $\mu(y,f_j,D)=\mu(y,f,D)$ for every  
$j\ge j_0$ and every $y\in E$. 
\end{lemma}

\begin{proof}
Let $\Omega\Subset f(D)$
be a domain containing $E$ and set $V:=f^{-1}(\Omega)\cap D$.
Then $(f|V)^*(\omega_{\Omega})\in H^n_c(V;\bZ)$.
Set $V_j:=f_j^{-1}(\Omega)\cap D$ for each $j\in\bN$. 
Since $f(\partial D)\cap \Omega = \emptyset$, 
by the uniform convergence of $(f_j)$ to $f$ on $\partial D$, 
there exists $j_0\in \bN$ for which $f_j(\partial D)\cap \Omega = \emptyset$ 
for every $j\ge j_0$. 
Thus $(f_j|V_j)^*(\omega_{\Omega})\in H^n_c(V_j;\bZ)$ for $j\ge j_0$ .
Furthermore, mappings $f|D$ and $f_j|D$ are properly homotopic with respect to $\Omega$ for every $j\in \bN$ large enough, that is, there exists $j_1\in \bN$ so that for every $j\ge j_1$ there exists a homotopy $F_j \colon  
\overline{D} \times [0,1]\to N$ from $f|\overline{D}$ to  
$f_j|\overline{D}$ and $F_j(\partial D \times [0,1])\cap \Omega = \emptyset$. 
Thus $\iota_{V,D}((f|V)^*\omega_{\Omega})=\iota_{V_j,D}((f_j|V_j)^*\omega_{\Omega})$ for $j\ge \max\{j_0,j_1\}$,
and \eqref{eq:cohomology} completes the proof.
\end{proof}

A point $x'\in M$ is a {\itshape non-normality point} of
a family $\mathcal{F}$ of $K$-quasiregular mappings from $M$ to $N$
if $\mathcal{F}$ is not normal on any open neighborhood of $x'$.
A point $x'\in M$ is an {\itshape isolated essential singularity}
of a quasiregular mapping $f:M\setminus\{x'\}\to N$ 
if $f$ does not extend to a continuous mapping from $M$ to $N$.

From now on, suppose that $N$ is closed.
The following theorems are manifold versions Miniowitz's Zalcman-type lemma (\cite[Lemma 1]{Miniowitz82}) and a Miniowitz--Zalcman-type rescaling principle for isolated essential singularities, respectively.

\begin{theorem}[{\cite[Theorem 19.9.3]{Iwaniec-Martin-book}}]
\label{th:Miniowitz}
Let $M$ be an oriented Riemannian $n$-manifold and 
$N$ a closed and oriented Riemannian $n$-manifold, $n\ge 2$, and let $x'\in M$. 
Then a family $\cF$ of $K$-quasiregular mappings, $K\ge 1$,
from $M$ to $N$ is not normal at $x'$ if and only if 
there exist sequences $(x_j)$, $(\rho_j)$, and $(f_j)$ in
$\bR^n$, $(0,\infty)$, and $\mathcal{F}$, respectively, and a non-constant 
$K$-quasiregular mapping $g \colon \R^n\to N$ such that 
$\lim_{j \to \infty} x_j = \phi(x')$, $\lim_{j\to\infty}\rho_j=0$ and 
\begin{gather}
\lim_{j\to\infty}f_j\circ \phi^{-1}(x_j+\rho_j v)=g(v)\label{eq:Miniowitz}
\end{gather}
locally uniformly on $\bR^n$,
where $\phi \colon D \to \R^n$ is a coordinate chart of $M$ at $x'$.
\end{theorem}

\begin{theorem}[{\cite[Theorem 1]{OPrescaling}}]
\label{th:rescalesing}
Let $M$ be an oriented Riemannian $n$-manifold and 
$N$ a closed and oriented Riemannian $n$-manifold, $n\ge 2$, and let $x'\in M$. 
Then a $K$-quasiregular mapping $f:M\setminus\{x'\}\to N$, $K\ge 1$,
has an essential singularity at $x'$ if and only if 
there exist sequences $(x_j)$ and $(\rho_j)$ in $\bR^n$ and $(0,\infty)$, 
respectively, and a non-constant $K$-quasiregular mapping $g:X\to N$,
where $X$ is either $\R^n$ or $\R^n\setminus\{0\}$,
such that $\lim_{j\to\infty}x_j=\phi(x')$, $\lim_{j\to\infty}\rho_j=0$, and
\begin{gather}
 \lim_{j\to\infty}f\circ\phi^{-1}(x_j+\rho_j v)=g(v)\label{eq:YO}
\end{gather}
locally uniformly on $X$, where $\phi:D \to \R^n$ is a coordinate chart of $M$ at $x'$.
\end{theorem}

By the Holopainen--Rickman Picard-type theorem \cite{HR98}, 
{\itshape for every $n\ge 2$ and every $K\ge 1$,
there exists a non-negative integer $q$ such that
$\#(N\setminus f(\bR^n))\le q$
for every closed and oriented Riemannian $n$-manifold $N$ and
every non-constant $K$-quasiregular mapping $f\colon \bR^n \to N$.}
We use this Picard-type theorem in this article also in the following form.

\begin{theorem}\label{th:puncture}
For every $n\ge 2$ and every $K\ge 1$,
there exists a non-negative integer $q'$ such that
$\#(N\setminus g(X))\le q'$
for every closed and oriented Riemannian $n$-manifold $N$ and
every non-constant $K$-quasiregular mapping $f\colon X\to N$,
where $X$ is either $\bR^n$ or $\bR^n\setminus\{0\}$.
\end{theorem}

\begin{proof}
 Let $Z_n\colon \bR^n\to\bR^n\setminus\{0\}$ be the Zorich mapping and $K_n\ge 1$ the distortion constant of $Z_n$; see e.g.\;\cite[I.3.3]{Rickman93} for the construction of the Zorich map. Set $K':=K\cdot K_n\ge 1$. Replacing $f$ with $f\circ Z_n$ if necessary, we may assume that $f$ is a $K'$-quasiregular mapping
 from $\bR^n$ to $N$. Now the Holopainen--Rickman Picard-type theorem \cite{HR98} completes the proof.
\end{proof}

Let $q'(n,K)$ be the smallest such $q'\in\bN\cup\{0\}$
as in Theorem \ref{th:puncture}, which we call the 
\emph{quasiregular Picard constant for parameters $n\ge 2$ and $K\ge 1$}. 

Having a Hurwitz-type theorem (Lemma \ref{lemma:Hurwitz}) and
rescaling theorems for a non-normality point of a family of $K$-quasiregular mappings 
and for an essential isolated singularity
of a quasiregular mapping (Theorems \ref{th:Miniowitz} and \ref{th:rescalesing})
at our disposal,
a ``from little to big by rescaling'' argument deduces 
the following Montel-type and big Picard-type theorems;
see \cite{Miniowitz82} and \cite[Theorem 2]{OPrescaling}.

\begin{theorem}\label{th:Montel}
Let $M$ be an oriented Riemannian $n$-manifold and 
$N$ a closed and oriented Riemannian $n$-manifold, $n\ge 2$. 
Then a non-normality point $x'\in M$ of a family $\cF$ of $K$-quasiregular mappings,
$K\ge 1$, from $M$ to $N$ is contained in $\overline{\bigcup_{f\in\mathcal{F}}f^{-1}(y)}$ 
for every $y\in N$ except for at most $q'(n,K)$ points.
\end{theorem}

\begin{theorem}\label{th:bP}
Let $M$ be an oriented Riemannian $n$-manifold and 
$N$ a closed and oriented Riemannian $n$-manifold, $n\ge 2$.
Then an essential singularity $x'\in M$ 
of a $K$-quasiregular mapping $f:M\setminus\{x'\}\to N$, $K\ge 1$,
is accumulated by $f^{-1}(y)$ for every $y \in N$ except for at most $q'(n,K)$ points.
\end{theorem}

The similarity Theorems \ref{th:Montel} and \ref{th:bP} goes beyond the statements and we prove these results simultaneously. The argument can also be viewed as a prototype of the proofs of Theorems \ref{th:periodic} and \ref{th:repdense}.

\begin{proof}[Proof of Theorems $\ref{th:Montel}$ and $\ref{th:bP}$]
Let $x'\in M$ be either a non-normality point in Theorem \ref{th:Montel} 
or an isolated essential singularity in Theorem \ref{th:bP}.

Let $X$ is either $\bR^n$ or $\bR^n\setminus\{0\}$ and let $g\colon X\to N$ be the non-constant quasiregular mapping $v \mapsto f_j\circ\phi^{-1}(x_j+\rho_jv)$
as in Lemma \ref{th:Miniowitz} or in Lemma \ref{th:rescalesing}, respectively, 
associated to this $x'$. Here $f_j\equiv f$ if $x'$ is as in Lemma \ref{th:bP}. 

Then $g(X)$ is an open subset in $N$, and satisfies $\#(N\setminus g(X))\le q'(n,K)$ by Theorem \ref{th:puncture}. 

Let $y\in g(X)$. Fix a subdomain $U$ in $N$ containing $y$ for which
some component $V$ of $g^{-1}(U)$ is relatively compact in $X$.
Then $g\colon V\to U$ is proper. By the locally uniform convergence 
and Lemma \ref{lemma:Hurwitz},
for every $j\in\bN$ large enough, there exists $v_j\in V$ such that 
$\phi^{-1}(x_j+\rho_jv_j)\in f_j^{-1}(y)$. 
By the uniform convergence, $\lim_{j\to\infty}\phi^{-1}(x_j+\rho_jv)=x'$ uniformly 
on $v\in\overline{V}$.  Thus $\lim_{j\to\infty}\phi^{-1}(x_j+\rho_jv_j)=x'$ and
$x'\in\overline{\bigcup_{j\in\bN}f_j^{-1}(y)}$.

Moreover, if $x'$ is an essential singularity of $f$, 
then $\phi^{-1}(x_j+\rho_jv_j)\neq x'$ for every $j\in\bN$.
Thus $x'$ is accumulated by $\bigcup_{j\in\bN}f_j^{-1}(y)=f^{-1}(y)$.
\end{proof}

The following Nevanlinna's four totally ramified value theorem is
specific to the case $n=2$. Theorem \ref{th:Nevanlinna}
reduces to the original case that $X=\bR^2$ and $N=\bS^2$
by lifting it to the (conformal) universal coverings of $X$ and $N$, which are
isomorphic to $\bR^2$ and a subdomain in $\bS^2$, respectively.

\begin{theorem}[cf.\ {\cite[p.\ 279, Theorem]{Nevan70}}]\label{th:Nevanlinna}
 Let $g:X\to N$ be a non-constant quasiregular mapping from  
 $X$ to 
 a closed, oriented, and connected Riemannian $2$-manifold $N$,
 where $X$ is either $\R^2$ or $\R^2\setminus\{0\}$.
 Then for every $E\subset N$ containing more than $4$ points,
 $E\cap g(X\setminus B_g)\neq\emptyset$.
\end{theorem}

Again, having a Hurwitz-type theorem (Lemma \ref{lemma:Hurwitz}) and
rescaling theorems for both a non-normality point of a 
family of $K$-quasiregular mappings
and an isolated singularity of a quasiregular mapping 
(Theorems \ref{th:Miniowitz} and \ref{th:rescalesing})
at our disposal, a ``from little to big by rescaling'' argument 
deduces the following two big versions of Theorem \ref{th:Nevanlinna}.

\begin{lemma}\label{th:bigNevanfamily}
Let $M$ be an oriented Riemannian $2$-manifold and 
$N$ a closed and oriented Riemannian $2$-manifold, $n\ge 2$.
Then a non-normality point $x'\in M$ of
a family $\cF$ of $K$-quasiregular mappings, $K\ge 1$, from $M$ to $N$
is contained in $\overline{\bigcup_{f\in\mathcal{F}}(f^{-1}(E)\setminus B_f)}$ 
for every $E\subset N$ containing more than $4$ points.
\end{lemma}

\begin{lemma}\label{th:bigNevan}
Let $M$ be an oriented Riemannian $2$-manifold and 
$N$ a closed and oriented Riemannian $2$-manifold, $n\ge 2$.
Then an essential singularity $x'\in M$
of a quasiregular mapping $f\colon M\setminus\{x'\}\to N$
is accumulated by $f^{-1}(E)\setminus B_f$ 
for every $E\subset N$ containing more than $4$ points.
\end{lemma}

Again, due the similarity of the statements we give a simultaneous proof.

\begin{proof}[Proof of Lemmas $\ref{th:bigNevanfamily}$ and
$\ref{th:bigNevan}$]
Let $x'\in M$ be as in either Lemma $\ref{th:bigNevanfamily}$ or
Lemma $\ref{th:bigNevan}$,
and let $g(v)=f_j\circ\phi^{-1}(x_j+\rho_jv)$
be a non-constant quasiregular mapping from $X$ to $N$
as in Lemmas \ref{th:Miniowitz} and \ref{th:rescalesing}, respectively, 
associated to this $x'$, where $X$ is either $\bR^2$ or $\bR^2\setminus\{0\}$, and
$f_j\equiv f$ in the case that $x'$ is as in Lemma $\ref{th:bigNevan}$.

Let $E$ be a subset in $N$ containing more than $4$ points.
Then by Nevanlinna's four totally ramified values theorem (Theorem \ref{th:Nevanlinna}),
$g^{-1}(E)\setminus B_g\neq\emptyset$.
Fix subdomains $U$ in $N$ intersecting $E$ small enough that 
some component $V$ of $g^{-1}(U)$ is relatively compact in $X\setminus B_g$.
Then $g:V\to U$ is univalent, and 
by the locally uniform convergence \eqref{eq:Miniowitz} or \eqref{eq:YO} on $X$
and the Hurwitz-type theorem (Lemma \ref{lemma:Hurwitz}),
for every $j\in\bN$ large enough, there exists $v_j\in V$ such that 
$\phi^{-1}(x_j+\rho_jv_j)\in f_j^{-1}(E)\setminus B_{f_j}$. 
Furthermore, $\lim_{j\to\infty}\phi^{-1}(x_j+\rho_jv)=x'$ 
uniformly on $v\in\overline{V}$. Thus
$\lim_{j\to\infty}\phi^{-1}(x_j+\rho_jv_j)=x'$ and
$x'\in\overline{\bigcup_{j\in\bN}f_j^{-1}(E)\setminus B_{f_j}}$.

Moreover, in the case that $x'$ is as in Lemma $\ref{th:bigNevan}$, 
then $\phi^{-1}(x_j+\rho_jv_j)\neq x'$ for every $j\in\bN$,
so $x'$ is accumulated by $\bigcup_{j\in\bN}f_j^{-1}(E)\setminus B_{f_j}
=f^{-1}(E)\setminus B_f$.
\end{proof}

Let $f\colon \Omega\to M$ be a non-constant local uniformly $K$-quasiregular mapping 
from an open subset $\Omega$ in a closed and oriented Riemannian $n$-manifold $M$, 
$n\ge 2$, to $M$. 
The following lemmas are elementary.

\begin{lemma}\label{th:invariance}
$f^{-1}(\cE(f))\subset\cE(f), f^{-1}(D_f)\subset D_f, f(D_f)\subset D_f, 
f^{-1}(F(f))\subset F(f), f(F(f))\subset F(f), f^{-1}(J(f))\subset J(f)$, and $f(J(f)\cap D_f)\subset J(f)$.
\end{lemma}

\begin{proof}
 The first inclusion $f^{-1}(\cE(f))\subset\cE(f)$ is obvious.
 The inclusion $f^{-1}(D_f)\subset D_f$ immediately follows by the continuity and openness of $f$. 
 The inclusion $f(D_f)\subset D_f$ also follows by the continuity and openness of $f$.

 The inclusion $f^{-1}(F(f))\subset F(f)$
 follows by the continuity and openness of $f$ and the Arzel\`a-Ascoli theorem.
 Indeed, let $x\in f^{-1}(F(f))$. Then $\{f^k;k\in\bN\}$ is equicontinuous at $f(x)$, 
 so $\{f^k\circ f;k\in\bN\}$ is equicontinuous at $x$. Hence $x\in F(f)$.

 Similarly, the inclusion $f(F(f))\subset F(f)$
 also follows by the continuity and openness of $f$ and the Arzel\`a-Ascoli theorem.
 Indeed, let $x\in f(F(f))$, i.e., $x=f(y)$ for some $y\in F(f)$. Then
 $\{f^k\circ f;k\in\bN\}$ is equicontinuous at $y$, so
 $\{f^k;k\in\bN\}$ is equicontinuous at $x=f(y)$. Hence $x\in F(f)$.

 Let us show $f^{-1}(J(f))\subset J(f)$. The inclusion $f^{-1}(J(f)\setminus D_f)\subset J(f)$ follows from
 $f(D_f)\subset D_f$, which is equivalent to
 $f^{-1}(M\setminus D_f)\subset M\setminus D_f$, and $M\setminus D_f\subset J(f)$.
 The inclusion $f^{-1}(J(f)\cap D_f)\subset J(f)$ follows from 
 $J(f)\cap D_f=D_f\setminus F(f)$ and $f(F(f))\subset F(f)$.

 The final $f(J(f)\cap D_f)\subset J(f)$ follows from 
 $f^{-1}(F(f))\subset F(f)$,
 which implies $f(D_f\setminus F(f))\subset D_f\setminus F(f)$,
 and $J(f)\cap D_f=D_f\setminus F(f)$.
\end{proof}

\begin{lemma}\label{th:interior}
The interior of $J(f)\cap D_f$ is empty unless $J(f)=M$. 
\end{lemma}

\begin{proof}
Let $x\in J(f)$ be an interior point of $J(f)$, and fix an open neighborhood $U$ of $x$
in $M$ contained in $J(f)$. Then by the Montel-type theorem (Theorem \ref{th:Montel}),
we have $\#(M\setminus\bigcup_{k\in\bN}f^k(U))<\infty$, so 
$M=\overline{\bigcup_{k\in\bN}f^k(U)}$, which is in $J(f)$ by Lemma \ref{th:invariance}
and the closedness of $J(f)$.
\end{proof}

A cyclic Fatou component of $f$ is a component $U$ of $F(f)$ 
such that $f^p(U)\subset U$ for some $p\in\bN$, which is called a period of $U$
(under $f$).
The proof of the following
is almost verbatim to the Euclidean case 
and we refer to Hinkkanen--Martin--Mayer \cite[Proposition 4.9]{HMM04} for the details.

\begin{theorem}\label{th:classification}
Let $\Omega$ be an open subset in a closed and oriented Riemannian $n$-manifold $M$, 
$n\ge 2$, and 
$f\colon\Omega\to M$ be a non-elementary local uniformly quasiregular mapping. Then
a cyclic Fatou component $U$ of $f$ having a period $p\in\bN$ 
is one of the following:
\begin{itemize}
 \item[(i)] a singular $($or rotation$)$ domain of $f$, that is,
       $f^p\colon U\to f^p(U)$ is univalent and 
       the limit of any locally uniformly convergent sequence $(f^{pk_i})_i$ on $U$,
       where $\lim_{i\to\infty}k_i=\infty$, is non-constant, 
 \item[(ii)] an immediate attractive basin of $f$, that is, the sequence
       $(f^{pk})_k$ converges locally uniformly on $U$, 
       the limit is constant, and its value is in $U$, or
 \item[(iii)] an immediate parabolic basin of $f$, that is, 
       the limit of any locally uniformly convergent sequence
       $(f^{pk_i})_i$ on $U$,
       where $\lim_{i\to\infty}k_i=\infty$, is constant and its value is in $\partial U$. 
\end{itemize}
\end{theorem}

In the following sections, given a subset $S$ in $\bR^n$ and $a,b\in\bR$, we denote by $aS+b$ the set $\{av+b\in\bR^n;v\in S\}$.

\section{Proof of Theorem \ref{th:periodic}}
\label{sec:theorem1}

Let $\bM$ be a closed, oriented, and connected Riemannian $n$-manifold, 
$n\ge 2$, and 
$f \colon \bM\setminus S_f\to \bM$ be a non-constant local uniformly $K$-quasiregular mapping,
$K\ge 1$, where $S_f$ is a countable and closed subset in $\bM$ and
consists of isolated essential singularities of $f$ and their accumulation points in $\bM$.

\begin{lemma}
The interior of $J(f)$ is empty unless $J(f)=\bM$.
\end{lemma}

\begin{proof}
 By Lemma \ref{th:interior}, the interior of $J(f)\cap D_f$ is empty
 unless $J(f)=\bM$. On the other hand, 
 $J(f)\setminus D_f=\overline{\bigcup_{k\ge 0}f^{-k}(S_f)}$, 
 which is the closure of a countable subset in $\bM$,
 has no interior by the Baire category theorem.
\end{proof}

Set
\begin{gather*}
J_1(f):= J(f)\setminus\overline{\bigcup_{k\ge 0}f^{-k}(S_f)}= J(f)\cap D_f\quad\text{and}\\
J_2(f):=\bigcup_{k\ge 0}f^{-k}(\{x\in S_f \colon x\ \mathrm{is\ isolated\ in\ }S_f\}).
\end{gather*}

The forthcoming arguments in this and the next sections rest on the following observation 
on the density of $J_1(f)\cup J_2(f)$ in $J(f)$.

\begin{lemma}\label{th:dense}
The set $J_1(f)\cup J_2(f)$ is dense in $J(f)$. Furthermore,
\begin{itemize}
\item[(i)] if $\#\bigcup_{k\ge 0}f^{-k}(S_f)<\infty$, then 
$J_1(f)\cup J_2(f)=J(f)$ and $\#J_2(f)<\infty$;
\item[(ii)] if $\#\bigcup_{k\ge 0}f^{-k}(S_f)=\infty$, 
then $J_1(f)=\emptyset$ and $J(f)=\overline{J_2(f)}$.
\end{itemize}
\end{lemma}

\begin{proof}
The density in $S_f$ of isolated points of $S_f$ implies
$\overline{\bigcup_{k\ge 0}f^{-k}(S_f)}=\overline{J_2(f)}$, so
$J_1(f)\cup\overline{J_2(f)}=J(f)$.
If $\#\bigcup_{k\ge 0}f^{-k}(S_f)<\infty$, then 
$J_2(f)=\bigcup_{k\ge 0}f^{-k}(S_f)=\overline{J_2(f)}$, 
so $J(f)=J_1(f)\cup J_2(f)$ and $\#J_2(f)<\infty$.
If $\#\bigcup_{k\ge 0}f^{-k}(S_f)=\infty$, 
then by the Montel-type theorem (Theorem \ref{th:Montel}), 
we have $J_1(f)=\emptyset$, so $J(f)=J_1(f)\cup\overline{J_2(f)}=\overline{J_2(f)}$.
\end{proof}

The following is a simple application
of the rescaling theorems (Theorems \ref{th:Miniowitz} and \ref{th:rescalesing}) 
to points in the dense subset $J_1(f)\cup J_2(f)$ in $J(f)$. 
We leave the details to the interested reader.

\begin{lemma}\label{lem:g}
Let $a\in J_1(f)\cup J_2(f)$ and 
let $\phi:D \to \R^n$ be a coordinate chart of $\bM$ at $a$.
Then there exist 
\begin{itemize}
 \item[(i)]  sequences $(x_m)$ in $\R^n$ and $(\rho_m)$ in $(0,\infty)$,
	which respectively tend to $\phi(a)$ and $0$ as $m\to \infty$, 
 \item[(ii)] a sequence $(k_m)$ in $\bN$, which is constant when $a\in J_2(f)$, and
 \item[(iii)] a non-constant $K$-quasiregular mapping $g\colon X\to\bM$, where $X$ is either $\R^n$ or $\R^n\setminus\{0\}$, and
       $X=\bR^n$ when $a\in J_1(f)$,
\end{itemize}
such that
\begin{gather}
 \lim_{m\to\infty}f^{k_m}\circ\phi^{-1}(x_m+\rho_m v)=g(v)\label{eq:juliarescaling}
\end{gather}
locally uniformly on $X$.
\end{lemma}

We show the remaining assertions in Theorem \ref{th:periodic} in separate lemmas. We continue to use the notation $q'(n,K)$ introduced in Section \ref{sec:background}.

We first show both the non-triviality of the Julia set $J(f)$ 
and the finiteness of the exceptional set $\cE(f)$ for non-injective $f$.

\begin{lemma}\label{lem:finiteness}
If $S_f\neq\emptyset$, then $f$ is non-injective, $J(f)\neq\emptyset$, and 
$\#\cE(f)\le q'(n,K)$. If $S_f=\emptyset$ and $f$ is not injective, 
then $J(f)\neq\emptyset$, $\cE(f)\subset F(f)$, and $\#\cE(f)\le q'(n,K)$.
\end{lemma}

\begin{proof}
If $S_f\neq\emptyset$, then 
by the big Picard-type theorem (Theorem \ref{th:bP}), 
$f$ is not injective and $\#\cE(f)\le q'(n,K)$, 
and by the definition of $J(f)$, we have $\emptyset\neq S_f\subset
\bigcup_{k\ge 0}f^{-k}(S_f)\subset J(f)$.

From now on, suppose that $S_f=\emptyset$ and $f:\bM\setminus S_f \to\bM$ is non-injective. Then 
$\deg f\ge 2$. We show first that $J(f)\ne \emptyset$.
Indeed, suppose $J(f)=\emptyset$. Then, by compactness of $\bM$, 
there exists a sequence $(k_m)$ in $\bN$ tending to $\infty$ such that
$(f^{k_m})$ tends to a $K$-quasiregular endomorphism $h\colon \bM \to \bM$
uniformly on $\bM$. 
Then for every $m\in\bN$ large enough, $f^{k_m}$ is homotopic to $h$ 
and $\deg h = \deg(f^{k_m})= (\deg f)^{k_m}\to\infty$ as $m\to\infty$ by the homotopy invariance of the degree. 
This is a contradiction and $J(f)\ne \emptyset$.

We show now that $\cE(f)\subset F(f)$.
Let $a\in \cE(f)$. Since $\#\bigcup_{k\ge 0}f^{-k}(a)<\infty$,
$f$ restricts to a permutation of $\bigcup_{k\ge 0}f^{-k}(a)$. Thus
there exists $p\in \bN$ for which $f^p(a)=a$ and $i(a,f^p) = \deg(f^p)\ge 2$. 
Fix a local chart $\phi:D\to\bR^n$ at $a$ and identify $f^p$
with $\phi\circ f^p\circ\phi^{-1}$ in a neighborhood of $a':=\phi(a)$ where the composition is defined. Then
there exist a neighborhood $U$ of $a'$ and $C>0$ such that 
for every $k\in\bN$, $f^{pk}$ is a $K$-quasiregular mapping from $U$ onto its image,
and that for every $k\in\bN$ and every $x\in U$,
\begin{gather*}
 |f^{pk}(x)-f^{pk}(a')|\le C |x-a'|^{(i(a',f^p)^k/K)^{1/(n-1)}}
\end{gather*}
by \cite[Theorem III.4.7]{Rickman93} (see also \cite[Lemma 4.1]{HMM04}).
Then $\lim_{k\to\infty}f^{pk}=a'$
locally uniformly on $U$. Hence $a\in F(f)$.

Finally, we show $\#\cE(f)\le q'(n,K)$.
If $\#\cE(f)>q'(n,K)$, we may fix $A\subset\cE(f)$ 
such that $q'(n,K)<\# A<\infty$ 
and $A':=\bigcup_{k\ge 0}f^{-k}(A)\subset \cE(f)$. 
Then $q'(n,K)<\#A'<\infty$, and by the above description of each point in $\cE(f)$, 
$f^{-1}(A')=A'$. 
By $\#A'>q'(n,K)$ and Theorem \ref{th:Montel},
$J(f)\subset\overline{\bigcup_{k\in\bN}f^{-k}(A')}$, which contradicts that
$\overline{\bigcup_{k\in\bN}f^{-k}(A')}=\overline{A'}=A'\subset\cE(f)\subset F(f)$.
\end{proof}

We snow next the accumulation of the backward orbits under $f$
of non-exceptional points to $J(f)$ for non-injective $f$, which implies
the perfectness of $J(f)$ for non-elementary $f$.

\begin{lemma}\label{th:backward}
Suppose $f$ is not injective. Then, for every $z\in\bM\setminus\cE(f)$, 
each point in $J(f)$ is accumulated by $\bigcup_{k\ge 0}f^{-k}(z)$.
Moreover, if $f$ is non-elementary, then $J(f)$ is perfect.
\end{lemma}

\begin{proof}
Fix $a\in J_1(f)\cup J_2(f)$.
Let $g(v)=\lim_{m\to\infty}f^{k_m}\circ\phi^{-1}(x_m+\rho_m v)$ be a 
non-constant quasiregular mapping from $X$ 
to $\bM$ as in Lemma \ref{lem:g} associated to this $a$. Then
$\#(\bM\setminus g(X))<\infty$ by Theorem \ref{th:puncture}.

Fix $z\in\bM\setminus\cE(f)$. Then 
we can choose subdomains $U_1$ and $U_2$ in $g(X)$ intersecting $\bigcup_{k\in\bN}f^{-k}(z)$ and having pair-wise disjoint closures so that, for each $i\in\{1,2\}$, some component $V_i$ of $g^{-1}(U_i)$ is relatively compact in $X$.

For each $i\in\{1,2\}$, $g\colon V_i\to U_i$ is proper. By the locally uniform convergence \eqref{eq:juliarescaling} on $X$ and 
Lemma \ref{lemma:Hurwitz}, 
$f^{k_m}(\phi^{-1}(x_m+\rho_m V_i))$ intersects $\bigcup_{k\ge 0}f^{-k}(z)$ 
for every $m\in\bN$ large enough.
Thus, for $m$ large enough, we may fix $v_m^{(i)}\in V_i$ satisfying
$y_m^{(i)}:=\phi^{-1}(x_m+\rho_m v_m^{(i)})\in\bigcup_{k\ge 0}f^{-k}(z)$.

Let $i\in \{1,2\}$.
By the uniform convergence $\lim_{m\to\infty}\phi^{-1}(x_m+\rho_m v)=a$ on $v\in\overline{V_i}$,
we have $\lim_{m\to\infty}y_m^{(i)}=a$, and, by the uniform convergence \eqref{eq:juliarescaling} on $\overline{V_i}$,
we have $\bigcap_{N\in\bN}\overline{\{f^{k_m}(y_m^{(i)});k\ge N \}}\subset
g(\overline{V_i})=\overline{U_i}$. 
Since $\overline{U_1}\cap\overline{U_2}=\emptyset$, $\{y_m^{(1)},y_m^{(2)}\}\neq\{a\}$ for $m\in\bN$ large enough.

Hence any point $a\in J_1(f)\cup J_2(f)$ 
is accumulated by $\bigcup_{k\in\bN}f^{-k}(z)$, and so is
any point in $J(f)$ by Lemma \ref{th:dense}.

If $f$ is non-elementary, then choosing $z\in J(f)\setminus\cE(f)$, 
we obtain the perfectness of $J(f)$ 
by the former assertion and $f^{-1}(J(f))\subset J(f)$.
\end{proof} 

We record the following consequence of Lemmas \ref{th:dense}, \ref{lem:finiteness}, and \ref{th:backward} as a lemma.

\begin{lemma}\label{th:stronger}
For non-elementary $f$, $J(f)$ is perfect, $\cE(f)$ is finite,
and any point in $J(f)$ is accumulated by $(J_1(f)\cup J_2(f))\setminus\cE(f)$.
\end{lemma}

Finally, the following lemma completes the proof of Theorem \ref{th:periodic}.

\begin{lemma}\label{lem:density}
If $f$ is non-elementary, then
any point in $J(f)$ is accumulated by the set of
all periodic points of $f$.
\end{lemma}

\begin{proof}
Fix an open subset $U$ in $\bM$ intersecting $J(f)$.
Let $a\in(J_1(f)\cup J_2(f))\setminus\cE(f)$, and
let $g(v)=\lim_{m\to\infty}f^{k_m}\circ\phi^{-1}(x_m+\rho_m v)$ be a 
non-constant quasiregular mapping from $X$ to $\bM$ as in Lemma \ref{lem:g} associated to this $a$, where $X$ is either $\R^n$ or $\R^n\setminus\{0\}$
and
$\phi \colon D\to \R^n$ is a coordinate chart of $\bM$ at $a$.
By Lemma \ref{th:backward} and Theorem \ref{th:puncture}, 
\begin{gather*}
 (U\cap\bigcup_{k\ge 0}f^{-k}(a))\cap g(X)\neq\emptyset.
\end{gather*}
Hence we can choose $j_1\in\bN\cup\{0\}$ and a subdomain $D_1\Subset D$ containing $a$ 
such that some component $U_1$ of $f^{-j_1}(D_1)$
is relatively compact in $U$ and that
some component $V_1$ of $g^{-1}(U_1)$ is relatively compact in $X$. Then 
$f^{j_1}\circ g:V_1\to D_1$ is proper.

Choose an open neighborhood $W\Subset X$ of $\overline{V_1}$ small enough that
$f^{j_1}\circ g(W)\Subset D$.
By the uniform convergence $\lim_{m\to\infty}\phi^{-1}(x_m+\rho_mv)=a\in D_1$ on 
$v\in\overline{W}$
and the uniform convergence \eqref{eq:juliarescaling} on $\overline{W}$,
we can define a mapping $\psi:\overline{W}\to\R^n$ and mappings $\psi_m:\overline{W}\to\R^n$
for every $m\in \bN$ large enough by
\begin{gather*}
\begin{cases}
 \psi(v):=\phi\circ f^{j_1} \circ g(v)-\phi(a)\quad\text{and}\\
 \psi_m(v):=\phi\circ f^{j_1} \circ f^{k_m}\circ\phi^{-1}(x_m+\rho_m v) - (x_m+\rho_m v),
\end{cases}
\end{gather*}
so that $\lim_{m\to\infty}\psi_m=\psi$ uniformly on $\overline{W}$.

The limit $\psi\colon V_1\to\psi(V_1)$ is non-constant, quasiregular, and proper, and
satisfies $0\in\psi(V_1)$ by $a\in D_1=f^{j_1}(g(V_1))$.
Although for each $m\in\bN$ large enough,
$\psi_m\colon V_1\to\bR^n$ is not necessarily quasiregular, 
we have $\lim_{m\to\infty}\mu(0,\psi_m,V_1)=\mu(0,\psi,V_1)>0$ after applying Lemma \ref{lemma:Hurwitz} to $(\psi_m)$ and $\psi$ on $\overline{V_1}$.
Thus $0\in\psi_m(V_1)$.

Hence for every $m\in \bN$ large enough, there exists $v_m\in V_1$
such that $y_m:=\phi^{-1}(x_m+\rho_m v_m)$ 
is a fixed point of $f^{j_1}\circ f^{k_m}$. Hence also $f^{k_m}(y_m)$ is a fixed point of $f^{j_1}\circ f^{k_m}$.
By the uniform convergence \eqref{eq:juliarescaling} on $\overline{V_1}$,
we have $\bigcap_{N\in\bN}\overline{\{f^{k_m}(y_m);k\ge N\}}\subset
g(\overline{V_1})=\overline{U_1}\subset U$, so $f^{k_m}(y_m)\in U$
for every $m\in\bN$ large enough.

We conclude that $J(f)$ is in the closure of the set of all periodic points
of $f$, so the perfectness of $J(f)$ completes the proof.
\end{proof}

\section{Proof of Theorem \ref{th:repdense}}
\label{sec:final}

Let $\bM$ be a closed, oriented, and connected Riemannian $n$-manifold, $n\ge 2$. 
Suppose $f\colon \bM\setminus S_f\to\bM$ is a non-elementary
local uniformly $K$-quasiregular mapping, $K\ge 1$, 
where $S_f$ is a countable and closed subset in $\bM$ and
consists of isolated essential singularities of $f$ and their accumulation points in $\bM$.
We continue to use the notations $J_1(f)$ and $J_2(f)$
introduced in Section \ref{sec:theorem1}. 

We first show the first assertion of Theorem \ref{th:repdense}.
\begin{lemma}
If $F(f)$ is non-empty and connected, then
every point in $J(f)$ is accumulated by the set of periodic points of $f$ 
contained in $J(f)$.
\end{lemma}

\begin{proof}
By the assumption, $F(f)$ is a fixed cyclic Fatou component of $f$. We show first that $f$ is not univalent on $F(f)$. 

We consider three cases separately.
In the case $S_f\neq\emptyset$, by the big Picard-type theorem (Theorem \ref{th:bP}), 
for every $y\in F(f)$ except for at most finitely many points,
we have $\# f^{-1}(y)=\infty$.
In the case that $S_f=\emptyset$ and $B_f\cap F(f)=\emptyset$, we have $\deg f\ge 2$,
and also $f(B_f)\cap F(f)=\emptyset$ by $f^{-1}(F(f))\subset F(f)$. Thus
$\# f^{-1}(y)=\deg f\ge 2$ for every $y\in F(f)$.
Since $f^{-1}(F(f))\subset F(f)$, $f$ is not univalent on $F(f)$ in these two cases.

Suppose now that $S_f=\emptyset$ and $B_f\cap F(f)\neq\emptyset$.
By the classification of cyclic Fatou components (Theorem \ref{th:classification}), $F(f)$ is a fixed immediate either attractive or parabolic basin of $f$.
So all the periodic points constructed in Lemma \ref{lem:density},
but at most one, are in $J(f)=\bM\setminus F(f)$.
\end{proof}

Next, we give a useful criterion for the repelling density in $J(f)$. 

\begin{lemma}\label{th:balk}
Let $a\in (J_1(f)\cup J_2(f))\setminus\cE(f)$ and suppose that a non-constant quasiregular mapping $g$ in Lemma $\ref{lem:g}$ associated to this $a$
satisfies the unramification condition
\begin{gather}
 a\not\in \bigcup_{k\in\bN}f^k(B_{f^k})\quad\text{and}\quad 
 J(f)\cap g(X\setminus B_g)\neq\emptyset.\label{eq:unramify} 
\end{gather}
Then every point in $J(f)$ is accumulated by the set of
all repelling periodic points of $f$.
\end{lemma}

\begin{proof}
Let $a\in(J_1(f)\cup J_2(f))\setminus\cE(f)$
and let $g(v)=\lim_{m\to\infty}f^{k_m}\circ\phi^{-1}(x_m+\rho_m v)$ be a 
non-constant quasiregular mapping from $X$ to $\bM$ 
as in Lemma \ref{lem:g} associated to this $a$, 
where $\phi \colon D\to \R^n$ is a coordinate chart of $\bM$ at $a$, and
suppose that these $a$ and $g$ satisfy \eqref{eq:unramify}.

Fix an open subset $U$ in $\bM$ intersecting $J(f)$.
By Lemma \ref{th:backward} and $\#\cE(f)<\infty$, 
there exists $j_1\in\bN\cup\{0\}$ such that $(f^{-j_1}(a)\cap U)\setminus\cE(f)\neq\emptyset$.
By the latter condition in \eqref{eq:unramify},
$g(X\setminus B_g)$ is an open subset in $\bM$ intersecting $J(f)$.
Thus, by Lemma \ref{th:backward}, there exists $j_2\in\bN\cup\{0\}$
such that $f^{-j_2}((f^{-j_1}(a)\cap U)\setminus\cE(f))\cap g(X\setminus B_g)\neq\emptyset$. 
Hence by the first condition in \eqref{eq:unramify},
we can choose a subdomain $D_1\Subset D\setminus f^{j_1+j_2}(B_{f^{j_1+j_2}})$ 
containing $a$ such that
some component $U_1$ of $f^{-j_1}(D_1)$ is relatively compact in $U$ and
that some component $V_1$ of $g^{-1}(f^{-j_2}(U_1))$ is relatively compact in $X\setminus B_g$.
Then $f^{j_1+j_2}\circ g:V_1\to D_1$ is univalent.

By the same argument as
in the proof of Lemma \ref{lem:density},  
we may choose, for every $m\in \bN$ large enough, a point $v_m\in V_1$ such that $y_m:=\phi^{-1}(x_m+\rho_m v_m)$ 
is a fixpoint of $f^{j_1+j_2}\circ f^{k_m}$. 
By the uniform convergence \eqref{eq:juliarescaling} on $\overline{V_1}$,
we have $\bigcap_{N\in\bN}\overline{\{f^{j_2}\circ f^{k_m}(y_m);k\ge N\}}\subset
f^{j_2}(g(\overline{V_1}))=\overline{U_1}\subset U$. Thus
$f^{j_2}\circ f^{k_m}(y_m)\in U$ for every $m\in\bN$ large enough.

Moreover, by the locally uniform convergence \eqref{eq:juliarescaling} on $X$ and Lemma \ref{lemma:Hurwitz}, 
the mapping $v\mapsto f^{j_1+j_2}\circ f^{k_m}\circ\phi^{-1}(x_m+\rho_m v)$
is a univalent mapping from $V_1$ onto its image for every $m\in \bN$ large enough. Hence
\begin{gather*}
 f^{j_1+j_2}\circ f^{k_m}:\phi^{-1}(x_m+\rho_m V_1)\to
 f^{j_1+j_2}\circ f^{k_m}(\phi^{-1}(x_m+\rho_m V_1))
\end{gather*}
is univalent for $m\in \bN$ large enough. By the uniform convergence 
\begin{gather*}
 \lim_{m\to\infty}\phi^{-1}(x_m+\rho_mv)=a\in D_1=f^{j_1+j_2}\circ g(V_1) 
\end{gather*}
on $v\in\overline{V_1}$ and the uniform convergence \eqref{eq:juliarescaling} on $\overline{V_1}$,
\begin{gather*}
 \phi^{-1}(x_m+\rho_m V_1)\Subset f^{j_1+j_2}\circ f^{k_m}(\phi^{-1}(x_m+\rho_m V_1))
\end{gather*}
for every $m\in\bN$ large enough.
Hence for every $m\in \bN$ large enough, $y_m$ 
is a repelling fixed point of $f^{j_1+j_2}\circ f^{k_m}$. 

We conclude that $J(f)$ is in the closure of the set of all repelling periodic points of $f$, so 
the perfectness of $J(f)$ completes the proof.
\end{proof}

We show the latter assertion of Theorem \ref{th:repdense} under the conditions given there, 
separately. 

{\bfseries Condition (\ref{item:connected}).} 
Suppose $\#\bigcup_{k\ge 0}f^{-k}(S_f)<\infty$. 
Then by Lemmas \ref{th:dense} and \ref{lem:finiteness}, we have
$\#(J_2(f)\cup\cE(f))<\infty$ and $J_1(f)=J(f)\setminus J_2(f)$. 
Suppose also that $\dim J(f)\ge n-1$.
For every $k\in\bN$, $\dim f^k(B_{f^k})\le n-2$, and then
$\dim(\bigcup_{k\in \bN} f^k(B_{f^k}))\le n-2$
(\cite[\S 2.2, Theorem III]{HurewiczW:Dimt}).
Hence we can fix 
$a\in J(f)\setminus(J_2(f)\cup\cE(f)\cup\bigcup_{k\in \bN} f^k(B_{f^k}))
=J_1(f)\setminus(\cE(f)\cup\bigcup_{k\in\bN}f^k(B_{f^k}))$,
and let $g\colon \bR^n\to \bM$ be a non-constant quasiregular mapping as in Lemma \ref{lem:g}
associated to this $a$. Then $\dim g(B_g)\le n-2$,
so $J(f)\cap g(\bR^n\setminus B_g)\neq\emptyset$.

The unramification condition \eqref{eq:unramify} is satisfied by these $a$ and $g$,
and Lemma \ref{th:balk} completes the proof in this case.

{\bfseries Condition (\ref{item:rep}).}
Let $a$ be a repelling periodic point of $f$ having a period $p\in\bN$
in $D_f\setminus(\cE(f)\cup\bigcup_{k\in\bN}f^k(B_{f^k}))$.
Then $a\in (J(f)\setminus\cE(f))\cap D(f)=J_1(f)\setminus\cE(f)$.
Let $g(v)=\lim_{m\to\infty}f^{k_m}\circ\phi^{-1}(x_m+\rho_m v)$ be a non-constant
quasiregular mapping from $\R^n$ to $\bM$ as in Lemma \ref{lem:g} associated to this $a$,
where $\phi\colon D\to \R^n$ is a coordinate chart of $\bM$ at this $a$.
By 
\cite[Theorem 6.3]{HMM04}, we may, in fact, 
assume that $x_m\equiv\phi(a)$ and $p|k_m$ for all $m\in \bN$, and
$g$ is in this case usually called a \emph{Koenigs mapping of $f^p$ at $a$}.
Then $g(0)=a$, and by the proof of \cite[Theorem 6.3]{HMM04}, we also have $0\not \in B_g$. 
Hence $a\in J(f)\cap g(\R^n\setminus B_g)$,  
and \eqref{eq:unramify} is satisfied by these $a$ and $g$.
Lemma \ref{th:balk} completes the proof in this case.

{\bfseries Condition (\ref{item:pcs}).}
Suppose that $J(f)\not\subset\bigcap_{j\in\bN}
\overline{\bigcup_{k\ge j}f^k(B_{f^k})}$.
By the closedness of $\bigcap_{j\in\bN}
\overline{\bigcup_{k\ge j}f^k(B_{f^k})}$ and Lemma \ref{th:stronger},
we indeed have $J(f)\not\subset(\cE(f)\cup\bigcap_{j\in\bN}
\overline{\bigcup_{k\ge j}f^k(B_{f^k})})$.
Hence we can fix $N\in\bN$ so large that the open subset
$U_N:=\bM\setminus(\cE(f)\cup\overline{\bigcup_{k\ge N}f^k(B_{f^k})})$
in $\bM$ intersects $J(f)$.

Let $a\in (J_1(f)\cup J_2(f))\cap U_N\subset(J_1(f)\cup J_2(f))\setminus\cE(f)$,
and let $g(v)=\lim_{m\to\infty}f^{k_m}\circ\phi^{-1}(x_m+\rho_mv)$ 
be a non-constant quasiregular mapping from $X$
to $\bM$
as in Lemma \ref{lem:g} associated to this $a$. 
Then $\#(\bM\setminus g(X))<\infty$ by Theorem \ref{th:puncture}.
We claim that $\#\bigcup_{k\ge N}f^{-k}(a)=\infty$. Indeed,
in the case $\#\bigcup_{k=0}^{N-1}f^{-k}(a)<\infty$, this follows by $a\not\in\cE(f)$.
In the case $\#\bigcup_{k=0}^{N-1}f^{-k}(a)=\infty$, we have $S_f\neq\emptyset$.
By applying the big Picard-type theorem (Theorem \ref{th:bP}) in at most $N$ times, we obtain $\#f^{-N}(a)=\infty$. Hence
we can fix $j_1\ge N$ such that $f^{-j_1}(a)\cap g(X)\neq\emptyset$, and 
a subdomain $U\Subset U_N$ containing $a$ so small that
some component $V$ of $(f^{j_1}\circ g)^{-1}(U)$ is relatively compact in $X$.
Then $g\colon V\to g(V)$ is proper.

By the uniform convergence \eqref{eq:juliarescaling} on $\overline{V}$,
for every $m\in\bN$ large enough, 
$f^{j_1}\circ f^{k_m}\circ\phi^{-1}(x_m+\rho_m V)\Subset U_N$.
Then by $j_1\ge N$ and the definition of $U_N$,
$f^{k_m}:\phi^{-1}(x_m+\rho_m V)\to f^{k_m}(\phi^{-1}(x_m+\rho_m V))$ is univalent, 
so the mapping $v\mapsto f^{k_m}\circ\phi^{-1}(x_m+\rho_m v)$ from $V$
onto its image is univalent. Hence
by the locally uniform convergence \eqref{eq:juliarescaling} on $X$ and
the Hurwitz-type theorem (Lemma \ref{lemma:Hurwitz}),  
$V\cap B_g=\emptyset$. Then 
$\emptyset\neq f^{-j_1}(a)\cap g(V)\subset J(f)\cap g(X\setminus B_g)$,
and \eqref{eq:unramify} is satisfied by these $a$ and $g$.
Lemma \ref{th:balk} completes the proof in this case.

{\bfseries Condition (\ref{item:surface}).}
Suppose that $n=2$. 
If $\#\bigcup_{k\ge 0}f^{-k}(S_f)<\infty$, then by Lemmas \ref{th:dense} and \ref{th:stronger},
$J_1(f)=J(f)\setminus J_2(f)$ is uncountable. Since
$\#\cE(f)<\infty$ (in Lemma \ref{th:stronger})
and $\bigcup_{k\ge 0}B_{f^k}$ is countable (when $n=2$),
we may fix $a\in J_1(f)\setminus(J_2(f)\cup\cE(f)\cup\bigcup_{k\in\bN}f^k(B_{f^k}))\subset
J_1(f)\setminus\cE(f)$. Let $g\colon \R^n\to\bM$ be a 
non-constant quasiregular mapping as in Lemma \ref{lem:g} associated to this $a$. 
By the countability of $B_g$ (when $n=2$) and the uncountability of $g^{-1}(J(f))$,
we also have $g^{-1}(J(f))\not\subset B_g$. 
The unramification condition \eqref{eq:unramify} is satisfied by these $a$ and $g$,
and Lemma \ref{th:balk} completes the proof in this case.

In the remaining case $\#\bigcup_{k\ge 0}f^{-k}(S_f)=\infty$, 
the argument similar to the above does not work.
For $n=2$, instead of Lemma \ref{th:balk},
we rely on the big versions (Lemmas \ref{th:bigNevanfamily} and \ref{th:bigNevan}) 
of the Nevanlinna four totally ramified value theorem (Theorem \ref{th:Nevanlinna})
to show Theorem $\ref{th:repdense}$ under $n=2$,
which is independent of the above proof specific to the case $\#\bigcup_{k\ge 0}f^{-k}(S_f)<\infty$.

\begin{proof}[Proof of Theorem $\ref{th:repdense}$ under $n=2$]
Set 
\begin{gather*}
 J'(f)
:=\begin{cases}
J_1(f)\setminus\{\text{all periodic points of }f\} & \text{if }
\#\bigcup_{k\ge 0}f^{-k}(S_f)<\infty,\\
J_2(f) & \text{if }\#\bigcup_{k\ge 0}f^{-k}(S_f)=\infty.
	\end{cases}
\end{gather*}

We claim that $J'(f)$ is dense in $J(f)$. If $\#\bigcup_{k\ge  
0}f^{-k}(S_f)=\infty$, we have $J(f)=\overline{J_2(f)}=\overline{J'(f)}$ by  
Lemma \ref{th:dense}. Thus we may assume that $\#\bigcup_{k\ge 0}f^{-k}(S_f)<\infty$ and it suffices to show that $J(f)=\overline{J'(f)}$.

By Lemmas \ref{th:dense} and \ref{th:stronger}, the set $J_1(f)$ is  
uncountable. Since $f$ has at most countably many periodic points,  
$J'(f)$ is non-empty. Let $y\in J'(f)$. If  
$J(f)\not\subset\overline{J'(f)}$, then every point in  
$J(f)\setminus\overline{J'(f)}$ is accumulated by $\bigcup_{k\ge  
0}f^{-k}(y)$ by Lemma \ref{th:backward}. On the other hand, by Lemma  
\ref{th:dense}, $\#J_2(f)<\infty$. Since $J_1(f)=J(f)\setminus  
J_2(f)$, there exists $x\in\bigcup_{k\ge  
0}f^{-k}(y)\cap(J_1(f)\setminus\overline{J'(f)})$. Thus $x$ is a  
periodic point of $f$, and so is $y$, which is a contradiction. Hence  
$J(f)=\overline{J'(f)}$ in the case $\#\bigcup_{k\ge 0}f^{-k}(S_f)<\infty$.

Since $J(f)$ is perfect, $\#J'(f)=\infty$. 
Fix an open subset $U$ in $\bM$ intersecting $J(f)$.
We claim that there exists $a\in J'(f)$
such that $\#(U\cap\bigcup_{k\ge 0}(f^{-k}(a)\setminus B_{f^k}))=\infty$.
Indeed, let $E\subset J'(f)$ such that $4<\# E<\infty$ and let $b'\in U\cap(J_1(f)\cup J_2(f))$.
For $b'\in J_1(f)$, $\{f^k;k\ge N\}$ is not normal at $b'$ for any $N\in\bN$.
Hence $b'\in\bigcap_{N\in\bN}\overline{\bigcup_{k\ge N}(f^{-k}(E)\setminus B_{f^k})}$ by Lemma \ref{th:bigNevanfamily}. 
Moreover, if $b'\in f^{-k}(E)$ for infinitely many $k\in\bN$,
then, by $\# E<\infty$, $f^{k_1}(b')=f^{k_2}(b')\in E$ for some $k_1<k_2$. Thus $f^{k_1}(b')\in E$ is a periodic point of $f$,
which contradicts $E\subset J'(f)$.
Hence $b'$ is accumulated by 
$\bigcup_{k\ge 0}(f^{-k}(E)\setminus B_{f^k})$.
In the case $b'\in J_2(f)$, $b'$ is an isolated essential singularity of 
$f^{j_1}$ for some $j_1\in\bN$, so
by Lemma \ref{th:bigNevan}, 
$b'$ is accumulated by $f^{-j_1}(E)\setminus B_{f^{j_1}}$.
In both cases, by $\# E<\infty$, we can choose $a\in E$ such that
$\#(U\cap\bigcup_{k\ge 0}(f^{-k}(a)\setminus B_{f^k}))=\infty$.

Let $g(v)=f^{k_m}\circ\phi^{-1}(x_m+\rho_mv)$
be a non-constant quasiregular mapping from $X$ to $\bM$
as in Lemma \ref{lem:g} associated to this $a$,
where $X$ is either $\R^2$ or $\R^2\setminus\{0\}$
and $\phi:D\to \R^2$ is a coordinate chart of $\bM$ at $a$. 
Then by the Nevanlinna four totally ramified value theorem 
(Theorem \ref{th:Nevanlinna}), 
\begin{gather*}
 \left(U\cap\bigcup_{k\ge 0}(f^{-k}(a)\setminus B_{f^k})\right)\cap g(X\setminus B_g)\neq\emptyset.
\end{gather*}
Hence we can choose $j_1\in\bN\cup\{0\}$
and a subdomain $D_1\Subset D$ containing $a$ such that
some component $U_1$ of $f^{-j_1}(D_1)$ is relatively compact in $U\setminus B_{f^{j_1}}$ and
that some component $V_1$ of $g^{-1}(U_1)$ is relatively compact in $X\setminus B_g$.
Then $f^{j_1}\circ g:V_1\to D_1$ is univalent.

By the same argument 
in the proof of Lemma \ref{lem:density}, for every $m\in \bN$ large enough, 
we can choose $v_m\in V_1$ such that $y_m:=\phi^{-1}(x_m+\rho_m v_m)$ 
is a fixed point of $f^{j_1}\circ f^{k_m}$, and so is $f^{k_m}(y_m)$,
and we also have $f^{k_m}(y_m)\in U$ for every $m\in \bN$ large enough.

Moreover, 
by the locally uniform convergence \eqref{eq:juliarescaling} on $X$ and 
Lemma \ref{lemma:Hurwitz},
the mapping $v\mapsto f^{j_1}\circ f^{k_m}\circ\phi^{-1}(x_m+\rho_m v)$
is also a univalent mapping from $V_1$ onto its image for every $m\in \bN$ large enough.
Hence
\begin{gather*}
 f^{j_1}\circ f^{k_m}:\phi^{-1}(x_m+\rho_m V_1)\to
 f^{j_1}\circ f^{k_m}(\phi^{-1}(x_m+\rho_m V_1))
\end{gather*}
is univalent for $m\in \bN$ large enough. By the uniform convergence 
$\lim_{m\to\infty}\phi^{-1}(x_m+\rho_mv)=a\in D_1=f^{j_1}\circ g(V_1)$ 
on $v\in\overline{V_1}$ and the uniform convergence \eqref{eq:juliarescaling} on $\overline{V_1}$,
\begin{gather*}
 \phi^{-1}(x_m+\rho_m V_1)\Subset f^{j_1}\circ f^{k_m}(\phi^{-1}(x_m+\rho_m V_1)).
\end{gather*}
for every $m\in\bN$ large enough. Hence 
$y_m$ is a repelling fixed point of $f^{j_1}\circ f^{k_m}$ for every $m\in \bN$ large enough. 

We conclude that $J(f)$ 
is in the closure of the set of all repelling periodic points of $f$,
so the perfectness of $J(f)$ completes the proof.
\end{proof}

\section{On the non-injectivity and non-elementarity of $f$}
\label{sec:discussion}

In the setting of Theorem \ref{th:periodic}, we have the following result on the non-elementarity of non-injective UQR-mappings.
\begin{lemma}
Let $\bM$ and $f\colon \bM \setminus S_f \to \bM$ be as in Theorem \ref{th:periodic}. Suppose in addition that $f$ is non-injective. Then $f$ is non-elemenatary if either $S_f =\emptyset$ or $\#\bigcup_{k\ge 0}f^{-k}(S_f)>q'(n,K)$.
\end{lemma}

\begin{proof}
For $S_f=\emptyset$ the claim follows from Theorem \ref{th:periodic}. Suppose $\#\bigcup_{k\ge 0}f^{-k}(S_f)>q'(n,K)$. By the big Picard-type theorem (Theorem \ref{th:bP}), we have $\#\bigcup_{k\ge 0}f^{-k}(S_f)=\infty$. Thus,
by Lemma \ref{th:dense}, $J(f)=\overline{\bigcup_{k\ge 0}f^{-k}(S_f)}$.
Hence $J(f)\not\subset\cE(f)$ since $\#\cE(f)<\infty$.
\end{proof}


It seems an interesting problem whether a non-injective $f$ 
is always non-elementary. 
This is the case in holomorphic dynamics, i.e., the case that $\bM=\bS^2$ and $K=1$. Indeed, if $0<\#\bigcup_{k\ge 0}f^{-k}(S_f)\le q'(2,1)=2$, 
$f$ can be normalized to be either a transcendental entire function on $\bC$
or a holomorphic endomorphism of $\bC\setminus\{0\}$
having essential singularities at $0,\infty$, both of which are
known to be non-elementary.

\end{document}